\documentclass{amsart}

\newtheorem{theorem}{Theorem}[section]

\newtheorem{proposition}[theorem]{Proposition}
\newtheorem{corollary}[theorem]{Corollary}

\theoremstyle{definition}
\newtheorem{definition}[theorem]{Definition}
\newtheorem{example}[theorem]{Example}

\theoremstyle{remark}

\allowdisplaybreaks

\numberwithin{equation}{section}

\begin{document}

\title{Quasi-idempotent Rota-Baxter operators arising from quasi-idempotent elements}

\author{Run-Qiang Jian}
\address{Department of Mathematics, Dongguan University
of Technology, 1, Daxue Road, Songshan Lake, 523808, Dongguan, P.
R. China}
\email{jianrq@dgut.edu.cn}
\thanks{}


\date{}

\dedicatory{}


\begin{abstract}
In this short note, we construct quasi-idempotent Rota-Baxter
operators by quasi-idempotent elements and show that every finite
dimensional Hopf algebra admits nontrivial Rota-Baxter algebra
structures and tridendriform algebra structures. Several concrete
examples are provided, including finite quantum groups and
Iwahori-Hecke algebras.
\end{abstract}

\maketitle

\section{Introduction}

A Rota-Baxter algebra of weight $\lambda\in \mathbb{C}$ is an
associative algebra $R$ equipped with a linear endomorphism $P$,
called a Rota-Baxter operator, verifying
$$P(a)P(b)=P(aP(b))+P(P(a)b)+\lambda P(ab),\forall a,b\in
R.$$Sometimes we denote a Rota-Baxter algebra by the pair $(R,P)$
or the triple $(R,\cdot, P)$ if we need to emphasize the
multiplication $\cdot$. The notion of Rota-Baxter algebra was
first explicitly introduced by G.-C. Rota \cite{R} based on the
work of G. Baxter \cite{B}. From then on, mathematicians began to
study various aspects of these algebras and their applications
(cf. \cite{C}, \cite{G1}, \cite{GK}, \cite{EGP}). Besides their
own interest in mathematics, Rota-Baxter algebras have also many
important applications in mathematical physics. For example, they
play an essential role in the Connes-Kreimer theory on the Hopf
algebra approach to renormalization in perturbative quantum field
theory (cf. \cite{CK1}, \cite{CK2}, \cite{EGK1}, \cite{EGK2},
\cite{EMP}). They also appear in the study of Loday type algebras
(\cite{E}, \cite{EG}), pre-Lie algebras (\cite{AB}), and
pre-Poisson algebras (\cite{A}).

Undoubtedly, good examples make a theory more dynamic. They
provide intuitions and ideas for the further study of the subject.
In the theory of Rota-Baxter algebras, there are many examples in
some sense. Indeed, every algebra possesses at least a Rota-Baxter
operator since the identity map is actually such an operator.
However this is trivial and useless. For all these reasons, we
would like to find systematic constructions of nontrivial
Rota-Baxter algebras. Moreover, we would like to search
interesting examples related to other branches of mathematics and
mathematical physics, and then we can apply the tools from
Rota-Baxter algebras to the study of these subjects. The recent
article \cite{J} is such an attempt. There, we use algebras in the
category of Hopf modules to construct idempotent Rota-Baxter
algebras. As an application, we obtain Rota-Baxter algebra
structures on the positive part of a quantum group. On the other
hand, the size of a Rota-Baxter algebra is usually very large. In
other words, most of the known examples are infinite dimensional
(cf. \cite{G2}). It is difficult to compute. The main purpose of
this note is to give an explicit construction of finite
dimensional Rota-Baxter algebras. Our starting point is the
following easy observation. Let $A$ be an associative algebra and
$a$ be an idempotent element of $A$, i.e., $a^2=a$. We denote by
$l_a$ the operator of left multiplication by $a$. Then for any
$b,c\in A$, we have
\begin{eqnarray*}
l_a(l_a(b)c)+l_a(bl_a(c))-l_a(bc)&=&a^2bc+abac-abc\\[3pt]
&=&abc+abac-abc\\[3pt]
&=&abac\\[3pt]
&=&l_a(b)l_a(c).
\end{eqnarray*}
Therefore $l_a$ is an idempotent Rota-Baxter operator of weight
$-1$. The problem is that idempotent elements do not always exist
except the neutral element. In order to overcome this difficulty,
we consider the so called quasi-idempotent elements. Recall that
an element $a\in A$ is said to be quasi-idempotent if it satisfies
$a^2=k_a a$ for some $k_a\in \mathbb{C}$. This slight modification
enables one to construct plenty of nontrivial examples having
small size since such elements do exist for finite dimensional
Hopf algebras. Note that finite dimensional Hopf algebras contain
lots of significant examples, especially, Lusztig's small quantum
groups (\cite{L1}, \cite{L2}). This construction also contains
other important examples such as Iwahori-Hecke algebras. So we can
apply our construction to these examples and get interesting
Rota-Baxter algebras.

This note is organized as follows. In Section 2, we recall the
notion of quasi-idempotent element in an algebra and use it to
construct quasi-idempotent Rota-Baxter operators. From this
construction, we show that every finite dimensional Hopf algebra
admits nontrivial Rota-Baxter algebra structures and tridendriform
algebra structures. In Section 3, based on the constructions given
in Section 2, we provide several interesting examples, including
finite quantum groups and Iwahori-Hecke algebras.

\section{Constructions}
For simplicity, we fix our ground field to be the complex number
field $\mathbb{C}$ throughout this note. All the objects we
discuss are defined over $\mathbb{C}$ unless otherwise specified.

We first recall the notion of quasi-idempotent operator which is
introduced in \cite{AM}.

\begin{definition}Let $A$ be an associative algebra and $\lambda\in \mathbb{C}$. A linear endomorphism $\phi$ of $A$ is called a \emph{quasi-idempotent operator of weight $\lambda$} if $\phi^2=-\lambda \phi$. A nonzero element $\xi\in A$ is called a \emph{quasi-idempotent element of weight $\lambda$} if $\xi^2=-\lambda\xi$.\end{definition}

Now we use quasi-idempotent elements to construct quasi-idempotent
Rota-Baxter operators.

\begin{proposition}For a fixed quasi-idempotent element $\xi\in A$ of weight $\lambda$, we define $P_\xi:A\rightarrow A$ by $P_\xi(a)=\xi a$ for any $a\in A$. Then $P_\xi$ is a quasi-idempotent Rota-Baxter operator of weight $\lambda$ on $A$.\end{proposition}
\begin{proof}For any $a,b\in A$, we have \begin{eqnarray*}
P_\xi\big(P_\xi(a)b\big)+P_\xi\big(aP_\xi(b)\big)+\lambda P_\xi(ab)&=&\xi^2ab+\xi a\xi b+\lambda\xi ab\\[3pt]
&=&-\lambda\xi ab+\xi a\xi b+\lambda\xi ab\\[3pt]
&=&\xi a\xi b\\[3pt]
&=&P_\xi(a)P_\xi(b).
\end{eqnarray*}\end{proof}

Here the meaning of the weight of $P_\xi$ is twofold: the weight
of quasi-idempotent operators and the weight of Rota-Baxter
operators.\\

Rota-Baxter algebras are closely related to tridendriform
algebras.

\begin{definition}[\cite{LR}]Let $V$ be a vector space, and $\prec$, $\succ$ and $\cdot$ be three binary operations on $V$. The quadruple $(V,\prec, \succ, \cdot)$ is called a \emph{tridendriform algebra} if the following relations are
satisfied: for any $x,y,z\in V$,
\[\begin{split}(x\prec y)\prec z&=x\prec(y\ast z),\\[3pt]
(x\succ y)\prec z&=x\succ (y\prec z),\\[3pt]
(x\ast y)\succ z&=x\succ(y\succ z),\\[3pt]
(x\succ y)\cdot z&=x\succ(y \cdot z),\\[3pt]
(x\prec y)\cdot z&=x\cdot (y\succ z),\\[3pt]
(x\cdot y)\prec z&=x\cdot (y\prec z),\\[3pt]
(x\cdot y)\cdot z&=x\cdot (y \cdot z),\end{split}\]where $x\ast
y=x\prec y+x\succ y+ x\cdot y$.\end{definition}

Given a Rota-Baxter algebra $(R,\cdot, P)$ of weight $1$,  we
define $a\prec b=a\cdot P(b)$ and $a\succ b=P(a)\cdot b$ for
$a,b\in R$. Then $(R,\prec,\succ,\cdot)$ is a tridendriform
algebra (see e.g. \cite{E}).

\begin{corollary}Given a quasi-idempotent element $\xi$ of weight $\lambda\neq 0$ in an algebra $A$, the operations $\prec$, $\succ$, and $\cdot$ defined below endow a tridendriform algebra structure on $A$: for any $a,b\in A$,\[\begin{split} a\prec b&=\lambda^{-1}a\xi b,\\[3pt]
a\succ b&=\lambda^{-1}\xi a b,\\[3pt]
a\cdot b&=ab.\end{split}\]\end{corollary}

If $\lambda=0$, we can obtain a similar construction of dendriform algebra structures on $A$.\\

Notice that quasi-idempotent elements of weight $-1$ are just idempotent elements. So they are a generalization of idempotent elements. But as we mentioned in the introduction, quasi-idempotent elements always exist in finite dimensional Hopf algebras while idempotent elements do not. In order to provide quasi-idempotent elements, we first recall some standard notions and facts from Hopf algebras.\\

We always denote by $(H,\Delta,\varepsilon,S)$ a finite
dimensional Hopf algebra with coproduct $\Delta$, counit
$\varepsilon$, and antipode $S$ in the sequel. We adopt Sweedler's
notion for coalgebra: $\Delta(h)=\sum h_{(1)}\otimes h_{(2)}$ for
any $h\in H$. We denote by $H^\ast$ the linear dual of $H$, i.e.,
$H^\ast=\mathrm{Hom}_\mathbb{C}(H, \mathbb{C})$, and by $<,>$ the
usual contraction between $H^\ast$ and $H$. So
$<a^\ast,a>=a^\ast(a)$ for any $a^\ast \in H^\ast $ and $a\in H$.
Then $H^\ast$ is also a Hopf algebra with multiplication
determined by $<a^\ast b^\ast, a>=\sum <a^\ast, a_{(1)}><b^\ast,
a_{(2)}>$. Let $l_{a^\ast}$ be the endomorphism of $H^\ast$
defined by the left multiplication by $a^\ast$. Then there is a
unique element $x_H$ such that $$<a^\ast,
x_H>=\mathrm{Tr}(l_{a^\ast}),\ \forall a^\ast \in H^\ast,$$ where
$\mathrm{Tr}$ is the usual trace of endomorphisms.

The element $x_H$ has the following properties (for a proof, one
can see Proposition 10.7.6 in \cite{Ra}).

\begin{proposition}Under the notation above, we have $\varepsilon(x_H)=\dim H$ and $x_H^2=\varepsilon(x_H)x_H$.\end{proposition}

So $x_H$ is a quasi-idempotent element of weight $-\dim H$. Furthermore, we have other quasi-idempotent elements in $H$. An element $\Lambda\in H$ is called a \emph{left (resp. right) integral} for $H$ if $a\Lambda=\varepsilon(a)\Lambda$ (resp. $\Lambda a=\varepsilon(a)\Lambda$) for all $a\in H$. Obviously, nonzero integrals are quasi-idempotent elements. It is well-known that the spaces of left integrals and right integrals are one-dimensional (see e.g. Theorem 10.2.2 in \cite{Ra}). These two spaces do not coincide in general (see Example \ref{example1} below). And normally, the element $x_H$ is not an integral since $cx_H=\varepsilon(c)x_H$ only for cocommutative element $c$ (that means $\sum c_{(1)}\otimes c_{(2)}=\sum c_{(2)}\otimes c_{(1)}$).\\

By combining the discussions above, we see that $P_{x_H}$ and $P_\Lambda$ are Rota-Baxter operators on $H$. As a consequence, we have

\begin{theorem}Every finite dimensional Hopf algebra admits nontrivial Rota-Baxter algebra structures and tridendriform algebra structures.\end{theorem}

\section{Examples}
In this section, we exhibit our construction by four concrete
examples. The first three are from Hopf algebras and quantum
groups, and the last one is from Iwahori-Hecke algebras.

\begin{example}[Group algebras]Let $G$ be a finite group and $H=\mathbb{C}[G]$ its group algebra. Then $H$ is a Hopf algebra with coproduct, counit, and antipode defined respectively by $$\Delta(g)=g\otimes g, \ \varepsilon(g)=1,\ S(g)=g^{-1},\ \forall g\in G.$$Both of the spaces of left integrals and right integrals of $H$ are $\mathbb{C}\xi$ where $\xi=\sum_{g\in G}g$. Then the operator $P_\xi$ on $H$ defined by $P_\xi(h)=\varepsilon(h)\xi$ for $h\in H$ is a Rota-Baxter operator of weight $-|G|$.\end{example}

\begin{example}[Sweedler's four-dimensional Hopf algebra]\label{example1}Let $H$ be the algebra generated by two elements $x$ and $y$ subject to $$x^2=\mathbf{1}, \ y^2=0, \ yx=-xy.$$ Then $H$ is a four-dimensional algebra with a linear basis $\{\mathbf{1},x, y, xy\}$ (see e.g. \cite{K}). Moreover it is a Hopf algebra equipped with the following operations:\[\begin{split}
\Delta(x)=x\otimes x&,\ \Delta(y)=\mathbf{1}\otimes y+y\otimes x,\\[3pt]
\varepsilon(x)&=1,\ \varepsilon(y)=0,\\[3pt]
S(x)&=x,\ S(y)=xy.\end{split}\]

We first compute the element $x_H$. Denote by $\{f_1,f_2,f_3,f_4\}$ the dual basis of $\{\mathbf{1},x, y, xy\}$. The multiplication table of $H^\ast$ is:\begin{center}\begin{tabular}{|l|l|l|l|l|} \hline $\cdot$& $f_1$ &$f_2$
&$f_3$ & $f_4$\\\hline $f_1$&$f_1$&0&$f_3$ &0\\\hline
$f_2$&0&$f_2$&0&$f_4$\\\hline
$f_3$&0&$f_3$&0&0\\\hline
$f_4$&$f_4$&0&0&0\\\hline
\end{tabular}\end{center}Then we have $$\mathrm{Tr}(l_{f_1})=\mathrm{Tr}(l_{f_2})=2,\ \mathrm{Tr}(l_{f_3})=\mathrm{Tr}(l_{f_4})=0,$$and hence $x_H=2(\mathbf{1}+x)$. By the technique of comparison of coefficients, one can show that the spaces of left integrals and right integrals are $\mathbb{C}(y+xy)$ and $\mathbb{C}(y-xy)$ respectively. In general, quasi-idempotent elements of $H$ are of the form $\xi=\mu_1(\mathbf{1}+x)+\mu_2y+\mu_3xy$ for some $\mu_1,\mu_2,\mu_3\in\mathbb{C}$. Therefore, the operator $P_\xi$ defined by the following actions is a Rota-Baxter operator of weight $-(2\mu_1+\mu_2+\mu_3)$: \[\begin{split}
P_\xi(\mathbf{1})&=\mu_1(\mathbf{1}+x)+\mu_2y+\mu_3xy,\\[3pt]
P_\xi(x)&=\mu_1(\mathbf{1}+x)-\mu_3y-\mu_2xy,\\[3pt]
P_\xi(y)&=\mu_1(y+xy),\\[3pt]
P_\xi(xy)&=\mu_1(y+xy).\end{split}\]\end{example}

\begin{example}[Lusztig's small quantum groups]Let $\mathfrak{g}$ be a complex simple Lie algebra. In \cite{L1} and \cite{L2}, Lusztig constructed some finite dimensional Hopf algebras attached to $\mathfrak{g}$ which are called small quantum groups or Frobenius-Lusztig kernels nowadays. Numerous mathematicians have investigated these interesting Hopf algebras. Here we endow Rota-Baxter algebra structures on these small quantum groups. For simplifying statements and computations, we only give the explicit formulas for $\mathfrak{g}=\mathfrak{sl}(2)$. We follow the description of the small quantum group attached to $\mathfrak{sl}(2)$ given in \cite{K}. Let $d$ be a positive integer $>2$ and $q$ be a $d$-th primitive root of 1. Denote \[e=\left\{
\begin{array}{lll}
d,&&\text{if }d\text{ is odd},\\[3pt]
d/2,&&\text{if }d\text{ is even}.
\end{array} \right.
\]We define $\overline{U}_q$ to be the algebra generated by $E,F,K,K^{-1}$ subject to \[\begin{split}
&KK^{-1}=K^{-1}K=1,\\[3pt]
KE&=q^2EK,\ KF=q^{-2}FK,\\[3pt]
&EF-FE=\frac{K-K^{-1}}{q-q^{-1}},\\[3pt]
&E^e=F^e=0,\ K^e=1.\end{split}\]Furthermore, $\overline{U}_q$ is a Hopf algebra with coproduct, counit, and antipode defined below:\[\begin{split}
\Delta(E)&=1\otimes E+E\otimes K,\ \Delta(F)=K^{-1}\otimes F+F\otimes 1,\\[3pt]
&\Delta(K)=K\otimes K,\ \Delta(K^{-1})=K^{-1}\otimes K^{-1},\\[3pt]
&\varepsilon(E)=\varepsilon(F)=0,\ \varepsilon(K)=\varepsilon(K^{-1})=1,\\[3pt]
S(E)=-&EK^{-1},\ S(F)=-KF,\ S(K)=K^{-1},
S(K^{-1})=K.\end{split}\]As a vector space, $\overline{U}_q$ is
finite dimensional with a basis $\{E^iF^jK^l\}_{0\leq i,j,l\leq
e-1}$. One can verify directly that
$$\xi=E^{e-1}F^{e-1}(1+K+K^2+\cdots+K^{e-1})$$is a left integral
of $\overline{U}_q$. Then the operator $P_{\xi}$ on
$\overline{U}_q$ of left multiplication by $\xi$ is a Rota-Baxter
operator of weight 0. We illustrate $P_\xi$ by its action on the
generators: \[\begin{split}
P_\xi(E)=\sum_{n=0}^{e-1}\frac{q^{2n+e-2}}{q-q^{-1}}[e-1]E^{e-1}F^{e-2}K^{n+1}&-\sum_{n=0}^{e-1}\frac{q^{2n+2-e}}{q-q^{-1}}[e-1]E^{e-1}F^{e-2}K^{n-1},\\[3pt]
P_\xi(F)&=0,\\[3pt]
P_\xi(K)&=\xi,\end{split}\]where $[e-1]=\frac{q^{e-1}-q^{1-e}}{q-q^{-1}}$.\end{example}

\begin{example}[Iwahori-Hecke algebras]Let $q$ be an indeterminate and $\mathcal{A}$ be the ring $\mathbb{Z}[q^{\frac{1}{2}},q^{-\frac{1}{2}}]$ of Laurent polynomials in $q^{\frac{1}{2}}$. Let $(W,S)$ be a Coxeter system and $\mathcal{H}$ be the corresponding Iwahori-Hecke algebras over $\mathcal{A}$. Then as an $\mathcal{A}$-module $\mathcal{H}$ has a standard basis $\{T_w|w\in W\}$, and the multiplication relations for this basis are: \[\begin{split}
T_wT_{w'}&=T_{ww'},\text{ if }l(ww')=l(w)+l(w'),\\[3pt]
T_s^2&=(q-1)T_s+q, \text{ for } s\in S,\end{split}\]where $l$ is
the usual length function on $W$. For any $s\in S$ we denote
$C_s=q^{-\frac{1}{2}}(T_s-q)$. They are parts of the
Kazhdan-Lusztig basis of $\mathcal{H}$ (see \cite{KL} and
\cite{H}). Then we have\begin{eqnarray*}
C_s^2&=&\big(q^{-\frac{1}{2}}(T_s-q)\big)^2\\[3pt]
&=&q^{-1}(T_s^2-2qT_s+q^2)\\[3pt]
&=&q^{-1}((q-1)T_s+q-2qT_s+q^2)\\[3pt]
&=&-(q^{\frac{1}{2}}+q^{-\frac{1}{2}})q^{-\frac{1}{2}}(T_s-q)\\[3pt]
&=&-(q^{\frac{1}{2}}+q^{-\frac{1}{2}})C_s.
\end{eqnarray*}So $C_s$ is an quasi-idempotent element of weight $q^{\frac{1}{2}}+q^{-\frac{1}{2}}$, and hence the operator $P_{C_s}$ given below is a Rota-Baxter operator of weight $q^{\frac{1}{2}}+q^{-\frac{1}{2}}$ on $\mathcal{H}$:
\[P_{C_s}(T_w)=\left\{
\begin{array}{ll}
q^{-\frac{1}{2}}T_{sw}-q^{\frac{1}{2}}T_w,&\text{ if }l(sw)>l(w),\\[3pt]
q^{\frac{1}{2}}T_{sw}-q^{-\frac{1}{2}}T_w,&\text{ if }l(sw)<l(w).
\end{array} \right.
\]\end{example}

\section*{Acknowledgements}
This work was started during my visit to Chern Institute of
Mathematics in February 2016. I would like to thank its support
and hospitality. I am indebted to Prof. Chengming Bai for his kind
invitation and useful discussions which push me to consider the
construction in this note. I also would like to thank Prof. Li Guo
for sharing his passions on Rota-Baxter algebras and valuable
discussions on related topics. Finally I would like to thank the
referee for careful reading and useful comments, especially, for
pointing out the reference \cite{AM} to me. This work was
partially supported by National Natural Science Foundation of
China (Grant No. 11201067).

\bibliographystyle{amsplain}

\end{document}